\documentclass[12pt]{amsart}
\usepackage{amsmath}
\usepackage{amssymb}
\usepackage{amsfonts}
\usepackage{amsthm}
\usepackage{mathrsfs}
\usepackage[all]{xy}

\newcommand{\be}{\begin{equation}}
\newcommand{\ee}{\end{equation}}
\newtheorem{theorem}{Theorem}[section]

\newtheorem{proposition}[theorem]{Proposition}

\theoremstyle{definition}
\newtheorem{definition}[theorem]{Definition}

\theoremstyle{remark}
\newtheorem{remark}[theorem]{Remark}

\numberwithin{equation}{section}

\begin{document}

\title{Wedge modules for two-parameter quantum groups} % $U_{r,s}(sl_n)$}
\author{Naihuan Jing, Lili Zhang, Ming Liu*}
\address{Department of Mathematics, North Carolina State University, Raleigh, NC 27695, USA}
\email{jing@math.ncsu.edu}
\address{School of Sciences, South China University of Technology,
Guangzhou 510640, China}
\email{lilizhang0902@163.com}
\address{Chern Institute of Mathematics, Nankai University, Tianjin 300071, China}
\email{ming.l1984@gmail.com}

\thanks{{\scriptsize
\hskip -0.4 true cm MSC (2010): Primary: 17B30; Secondary: 17B68.
%Received: October 2012, Accepted: 18 January 2013.\\
$*$Corresponding author.
}}

\maketitle

\begin{abstract}
The Yang-Baxterization $R(z)$ of the trigonometric R-matrix is computed for the two-parameter
quantum affine algebra $U_{r,s}(\widehat{sl}_n)$. Using the fusion rule we construct
all fundamental representations of the quantum algebra $U_{r,s}({sl}_n)$ as $(r, s)$-wedge products of
the natural representation.
\end{abstract}

\section{Introduction}

Two-parameter general linear and special linear quantum groups were introduced by Takeuchi \cite{T} in 1990.
This new type of quantum groups was used in \cite{DPW} to incorporate two seemingly different types of
quantum general linear groups:
the usual quantum general linear groups from Drinfeld-Jimbo quantum algebras \cite{D, Jb} and Dipper-Donkin quantum linear
groups \cite{DD} connected with quantum $q$-Schur algebras.
Similar two-parametric quantum groups originating from exotic solutions of Yang-Baxter equations
were also studied in \cite{J} and they interpolate the usual quantum groups and
the degenerate cases from a special solution of the six vertex model. Earlier in \cite{R} Reshetikhin studied quasi-triangular Hopf algebras
from multi-parameter solutions of Yang-Baxter equations.
All these quantum groups can be viewed as quantum transformation groups \cite{J, Do} over
certain quantum planes in the sense of Manin \cite{M}.

In 2001 Benkart and Witherspoon \cite{BW1} investigated the two-parameter quantum group
in connection with the down-up algebras.
In \cite{BW2} Benkart and Witherspoon developed the two-parameter quantum groups corresponding to general linear
and special linear Lie algebras $\mathfrak{gl}_n$ and $\mathfrak{sl}_n$, and constructed the
 corresponding $R$-matrix and the quantum Casimir element. They further showed that these two algebras
can be realized as
Drinfeld doubles. In \cite{BW3} the representation theory
of two-parameter quantum groups $U_{r,s}(\mathfrak{gl}_n)$ and $U_{r,s}(\mathfrak{sl}_n)$ was studied and an explicit description
of the (r,s)-symmetric tensor space $S^2_{r,s}(V)$ and the R-matrix $R=R_{VV}$ was given.

It is well known that the fundamental representations play an important role in the representation theory of classical Lie algebras.
In the case of $\mathfrak{sl}_n$, the fundamental representations are just the wedge modules of the natural representation. The fundamental representations of the quantum group $U_q(\mathfrak{sl}_n)$
were first constructed by Rosso in \cite{Rosso}. Later the fundamental modules
were reconstructed using the fusion procedure \cite{KRS, KMN, JMO}
in connection with the quantum affine algebras $U_q(\widehat{\mathfrak{sl}}_n)$.
A natural question of the wedge modules arises for the two-parameter quantum groups. Although the tensor
modules of $U_{r,s}(\widehat{\mathfrak{sl}}_n)$ have been constructed in \cite{BW3} it is non-trivial to
pass down to the irreducible quotient modules, and the best method seems to be the fusion procedure.
In order to carry out the fusion procedure one first needs to find the $R$-matrix with spectral parameter
as in the one-parameter case.

In this paper we first use the Yang-Baxterization method of Ge-Wu-Xue \cite{GWX} to construct
 a spectral parameter dependent $R$-matrix for the two-parameter quantum algebra $U_{r,s}(\widehat{\mathfrak{sl}}_n)$
based on the braid group representation given by the Benkart-Witherspoon R-matrix.
We remark that our R-matrix $R(z)$ can also be viewed
as the R-matrix of two-parameter quantum affine algebra $U_{r,s}(\widehat{\mathfrak{sl}}_n)$ \cite{HRZ}.
Then we construct all the (r,s)-wedge modules of $U_{r,s}(\mathfrak{sl}_n)$ by
the fusion procedure.

The paper is organized as follows. In section 2, we give a brief introduction of two-parameter quantum groups $U_{r,s}(\mathfrak{gl}_n)$
and $U_{r,s}(\mathfrak{sl}_n)$ and recall the results given in \cite{BW3}. In section 3, we obtain an
R-matrix with spectral parameter which can be regarded as the R-matrix corresponding to the
two-parameter quantum affine algebra $U_{r,s}(\widehat{\mathfrak{sl}}_n)$ by using the Yang-Baxterization
method of Ge, Wu and Xue.
In section 4 we determine all $(r,s)$-wedge modules of $U_{r,s}(\mathfrak{sl}_n)$.

\section{Two-parameter quantum group $U_{r,s}(\mathfrak{sl}_n)$ and R-matrix}

We first recall the definition of the two-parameter quantum group $U_{r,s}(\mathfrak{sl}_n)$ and some
basics about their representations from \cite{BW3}. Let $\epsilon_1,\epsilon_2,...,\epsilon_n$
denote an orthonormal basis
of a Euclidean space $E$ with an inner product $\langle \, ,\, \rangle$. Let $\Pi=\{\alpha_j=\epsilon_j-\epsilon_{j+1}|j=1,2,...,n-1\}$ be the set of the simple roots of type $A_{n-1}$,
then $\Phi=\{\epsilon_i-\epsilon_j|1\leq i\neq j\leq n\}$ is the set of all roots.

We now fix two nonzero elements $r,s\in \mathbb{C}$ with $r\neq s$.

\begin{definition}
The two-parameter quantized enveloping algebra $U_{r,s}(\mathfrak{sl}_n)$ is the unital
associative algebra over $\mathbb{C}$ generated by $e_i,f_i,\omega_i,\omega_i'$, $1\leq i<n$ with
the following relations:

(R1) The generators $\omega_i$, $\omega_i'$ are invertible elements commuting with each other,

(R2)
$\omega_ie_j=r^{\langle \epsilon_i,\alpha_j\rangle}s^{\langle \epsilon_{i+1},\alpha_j\rangle}e_j\omega_i$
and
$\omega_if_j=r^{-\langle \epsilon_i,\alpha_j\rangle}s^{-\langle \epsilon_{i+1},\alpha_j\rangle}f_j\omega_i,$

(R3) $\omega_i'e_j=r^{\langle \epsilon_{i+1},\alpha_j\rangle}s^{\langle \epsilon_{i},\alpha_j\rangle}e_j\omega_i'$
and
$
\omega_i'f_j=r^{-\langle \epsilon_{i+1},\alpha_j\rangle}s^{-\langle \epsilon_{i},\alpha_j\rangle}f_j\omega_i',$

(R4) $[e_i,f_j]=\frac{\delta_{ij}}{r-s}(\omega_i-\omega_i')$,

(R5) $[e_i,e_j]=[f_i,f_j]=0, \mbox{if}\  |i-j|>1,$

(R6) $e_{i}^2e_{i+1}-(r+s)e_ie_{i+1}e_i+rse_{i+1}e_{i}^2=0$
and

$e_{i+1}^2e_{i}-(r+s)e_{i+1}e_{i}e_{i+1}+rse_{i+1}^2e_{i}=0$,

(R7) $f_{i}^2f_{i+1}-(r^{-1}+s^{-1})f_if_{i+1}f_i+r^{-1}s^{-1}f_{i+1}f_{i}^2=0$ and

$f_{i+1}^2f_{i}-(r^{-1}+s^{-1})f_{i+1}f_{i}f_{i+1}+r^{-1}s^{-1}f_{i+1}^2f_{i}=0$,
where $[\, ,\, ]$ is the usual commutator.

\end{definition}
\begin{remark}
When $r=q$, $s=q^{-1}$, the algebra modulo the ideal generated by the elements $\omega_j^{-1}-\omega_j'$, $1\leq j<n$,
is isomorphic to $U_q(\mathfrak{sl}_n)$.
\end{remark}
The algebra $U_{r,s}(\mathfrak{sl}_n)$ is a Hopf algebra under the coproduct $\Delta$ such that
$\omega_i$, $\omega_i'$ are group-like elements and other nontrivial coproducts, counits and antipodes
are given by:

$$\Delta(e_i)=e_i\otimes 1+\omega_i\otimes e_i, \Delta(f_i)=1\otimes f_i+f_i\otimes \omega_i',$$

$$\epsilon(e_i)=\epsilon(f_i)=0, S(e_i)=-\omega_i^{-1}e_i, S(f_i)=-f_i\omega_i'^{-1}.$$

The natural representation of $U=U_{r,s}(\mathfrak{sl}_n)$ can be easily described as follows.
Let $\Lambda=\mathbb{Z}\epsilon_1\oplus \mathbb{Z}\epsilon_2\oplus\cdots\oplus \mathbb{Z}\epsilon_n$ be the
weight lattice of $\mathfrak{gl}_n$, where $\epsilon_i$ are the orthonomal vectors
 as before, $Q=\mathbb{Z}\Phi$ the root lattice, and we denote
$Q_+=\sum_{i=1}^{n-1}\mathbb{Z}_{\geq 0}\alpha_i$. Assume $\Lambda$ is equipped with
a partial order in which $\nu\leq \lambda$ if and only if $\lambda-\nu \in Q_+$. For each $i<n$ define the fundamental weights
\begin{equation*}
\overline{\omega}_i=\epsilon_1+\cdots+\epsilon_i.
%\overline{\omega}_i=\epsilon_1+\cdots+\epsilon_i-\frac in\sum_{j=1}^n\epsilon_j.
\end{equation*}
Then $\Lambda_{\mathfrak{sl}}=\mathbb{Z}\overline{\omega}_1\oplus\cdots\oplus\mathbb{Z}\overline{\omega}_{n-1}$
is the weight lattice of $\mathfrak{sl}_n$.

Corresponding to $\lambda\in \Lambda_{\mathfrak{sl}}$, there is an algebra homomorphism $\widehat{\lambda}$ from the
subalgebra $U^0$ of $U_{r,s}(\mathfrak{sl}_n)$ generated by $\omega_i^{\pm 1}$, $\omega_i'^{\pm 1}$ to
$\mathbb{C}$ given by:
\begin{equation}
\widehat{\lambda}(\omega_i)=r^{\langle\epsilon_i,\lambda\rangle}s^{\langle\epsilon_{i+1},\lambda\rangle}, ~~~~~~~~~~~~~~~~~~~~~\qquad ~~~~~~~~~~~~~~~~~~~~~~\widehat{\lambda}(\omega_i')=r^{\langle\epsilon_{i+1},\lambda\rangle}s^{\langle\epsilon_i,\lambda\rangle}.
\end{equation}
When $rs^{-1}$ is not a root of unity, the homomorphisms $\widehat{\lambda}=\widehat{\mu}$ if and only if the corresponding weights $\lambda=\mu$ \cite{BW2}.
Let $M$ be a finite dimensional $U_{r,s}$-module, then
\begin{equation}
M=\bigoplus_{\chi} M_{\chi},
\end{equation}
where $\chi$ are algebra homomorphisms $\chi: U^0\mapsto \mathbb C^{\times}$ and
$M_{\chi}=\{v\in M| (\omega_i-\chi(\omega_i))^mv=(\omega_i'-\chi(\omega_i'))^mv=0, \, \mbox{for all $i$ and for some $m$}\}$ are the associated generalized eigenspaces. For brevity the homomorphisms $\chi$ are called {\it generalized weights} of $M$. When all generalized weights are of the form $\chi(\widehat{-\alpha})$ for a fixed $\chi$ and $\alpha$ varying in $Q_+$, we say $M$ is a {\it highest weight module} of weight $\chi$ and write
$M=M(\chi)$. Benkart and Witherspoon \cite{BW2, BW3} have shown that when $M$ is simple, all generalized weight subspaces are actually weight subspaces. Moreover,
if all generalized weights %spaces $M_{\chi}$
are the homomorphisms $\widehat{\lambda}$ coming from the usual weights $\lambda$, we will simplify the notation
and write $M_{\lambda}$ for $M_{\widehat{\lambda}}$, and similarly the highest weight module $M(\widehat{\lambda})$ will be simply denoted as $M(\lambda)$.

The theory of highest weight modules of $U_{r,s}(\mathfrak g)$ is quite similar to that of corresponding simple Lie algebra $\mathfrak g$. Besides the highest weight modules discussed above, one can also define the notion of
Verma modules $M(\lambda)$ \cite{BW3}. Benkart and Witherspoon \cite{BW3} have proved that all finite dimensional $U_{r,s}(\mathfrak{sl}_n)$-modules are realized as the simple quotients of some Verma modules. We will
denote by $V(\lambda)$ the simple quotient of the Verma module $M(\lambda)$.

Let $V$ be the $n$-dimensional vector space over $\mathbb{C}$ with basis $\{v_j|1\leq j\leq n\}$, and define
$E_{ij}\in End(V)$ such that $E_{ij}v_k=\delta_{jk}v_i$.
The natural representation is the $U_{r,s}(\mathfrak{sl}_n)$-module $V$ with the action given by:
\begin{equation*}
\begin{aligned}
&e_j=E_{j,j+1}, ~~f_j=E_{j+1,j},\\
&\omega_j=rE_{jj}+sE_{j+1,j+1}+\sum_{k\neq j,j+1} E_{kk},\\
&\omega_j'=sE_{jj}+rE_{j+1,j+1}+\sum_{k\neq j,j+1} E_{kk},\\
\end{aligned}
\end{equation*}
where $1\leq j\leq n-1$.
%It can be verified that relations (R1)-(R7) hold, so $V$ is the so called natural representation.

It is clear that
\begin{equation}\label{eq2.2}
\omega_iv_j=r^{\langle\epsilon_i,\epsilon_j\rangle}s^{\langle\epsilon_{i+1},\epsilon_j\rangle}v_j,
\end{equation}
\begin{equation}\label{eq2.3}
\omega_i'v_j=r^{\langle\epsilon_{i+1},\epsilon_j\rangle}s^{\langle\epsilon_{i},\epsilon_j\rangle}v_j,
\end{equation}
 for all $i,j$, so $v_j$ has weight $\epsilon_j=\epsilon_1-(\alpha_1+\cdots+\alpha_{j-1})$. Therefore $V=\bigoplus_{j=1}^{n}V_{\epsilon_j}$ is the weight decomposition, and it is a simple module of $U_{r,s}(\mathfrak{sl}_n)$.

%\subsection{Tensor powers of natural module and R-matrix}

In \cite{BW3} Benkart and Witherspoon studied the tensor powers of the natural representation $V$ of $U_{r,s}(\mathfrak{sl}_n)$ and the associated R-matrix.
Let $R=R_{VV}$ be the $R$-matrix associated to $V$ given by the coproduct, and for $1\leq i<k$, let $R_i$ be the $U_{r,s}(\mathfrak{sl}_n)$ module isomorphism on $V^{\otimes k}$ defined by
\begin{equation*}
R_i(w_1\otimes w_2\otimes\cdots\otimes w_k)=w_1\otimes\cdots\otimes R(w_i\otimes w_{i+1})\otimes w_{i+2}\otimes\cdots\otimes w_k.
\end{equation*}

Since $R=R_{VV}$ satisfies the Yang-Baxter equation, the braid relations hold:
\begin{equation*}
R_iR_{i+1}R_i=R_{i+1}R_iR_{i+1}  ~~~~~~   for ~~~~1\leq i<k,
\end{equation*}
The construction also implies that for $|i-j|\geq 2$,
\begin{equation*}
R_iR_j=R_jR_i.
\end{equation*}
The following result was given in \cite{BW3}.
\begin{proposition}\label{prop2.3}
Whenever $s\neq -r$, the $U_{r,s}(\mathfrak{sl}_n)$-module $V\otimes V$ decomposes into a direct sum of two
simple submodules, $S_{r,s}^2(V)$(the (r,s)-symmetric tensors), and $\Lambda_{r,s}^2(V)$
(the (r,s)--antisymmetric tensors). These modules are defined as follows:
\begin{description}
  \item[(i)] $S_{r,s}^2(V)$ is the span of $\{v_i\otimes v_i|1\leq i\leq n\}\cup \{v_i\otimes v_j+sv_j\otimes v_i|1\leq i<j\leq n\}$.
  \item[(ii)] $\Lambda_{r,s}^2(V)$ is the span of $\{v_i\otimes v_j-rv_j\otimes v_i|1\leq i<j\leq n\}$.
\end{description}
\end{proposition}

Consequently one has that
\begin{equation}\label{minimal polynomial}
R_i^{2}=(1-rs^{-1})R_i+rs^{-1}Id
\end{equation}
for all $1\leq i<k$.

\begin{proposition}\label{prop2.4}\cite{BW3}
The minimum polynomial of $R$ on $V\otimes V$ is $(t-1)(t+rs^{-1})$ if $s\neq -r$.
\end{proposition}

From Proposition \ref{prop2.4}, it follows that the action of $R$ on $V\otimes V$ is given as follows:
\begin{equation}\label{R-matrix}
R=\sum_{i=1}^nE_{ii}\otimes E_{ii}+r\sum_{i<j}E_{ji}\otimes E_{ij}+s^{-1}\sum_{i<j}E_{ij}\otimes E_{ji}+(1-rs^{-1})\sum_{i<j}E_{jj}\otimes E_{ii}.
\end{equation}

\section{Yang-Baxterization}

Starting from an $R$-matrix of the Yang-Baxter equation (YBE), Yang-Baxterization of $R$ recovers the
associated spectral parameter dependent $R$-matrix $R(z)$ satisfying the YBE:
$$
R_1(z)R_2(zw)R_1(w)=R_2(w)R_1(zw)R_2(z).
$$
When $R$ has two or three eigenvalues Ge, Wu and Xue \cite{GWX} gave an algorithm to compute its Yang-Baxterization. Suppose $R$ has two distinct eigenvalues $\lambda_1$, $\lambda_2$,
the Yang-Baxterization $R(z)$ can be computed by:
\begin{equation}\label{Yang-Baxterization}
R(z)=\lambda_2^{-1}R+z\lambda_1R^{-1}.
\end{equation}

Proposition \ref{prop2.4} shows that $R=R_{VV}$ has eigenvalues 1 and $-rs^{-1}$ on $V\otimes V$,
then by the Yang-Baxterization we have the following theorem. We remark that the resulting R-matrix with spectral parameter $z$ is an $(r,s)$-analogue of the $R$-matrix for the quantum affine algebra $U_q(\widehat{sl}_n)$.
\begin{theorem}
For the braid group representation $R=R_{VV}$, the R-matrix $R(z)$ is given by
\begin{align}\nonumber
R(z)&=(1-zrs^{-1})\sum_{i=1}^nE_{ii}\otimes E_{ii}+(1-z)(r\sum_{i>j}+s^{-1}\sum_{i<j})E_{ij}\otimes E_{ji}\\ \label{R(z)}
&+z(1-rs^{-1})\sum_{i<j}E_{ii}\otimes E_{jj}+(1-rs^{-1})\sum_{i>j}E_{ii}\otimes E_{jj}.
\end{align}
\end{theorem}
\begin{proof} We already mentioned that $R=R_{VV}$ has two distinct eigenvalues $\lambda_1=-rs^{-1}$ and $\lambda_2=1$. Ge-Wu-Xue's Yang-Baxterization implies that the resulting R-matrix with spectral parameter is then given by:
\begin{equation*}
R(z)=R-zrs^{-1}R^{-1}
\end{equation*}
On the other hand, Eq. (\ref{minimal polynomial}) implies that
\begin{equation*}\label{Inverse}
R^{-1}=r^{-1}sR+(1-r^{-1}s) I.
\end{equation*}
Thus $R(z)=(1-z)R-z(rs^{-1}-1)I$, which is exactly Eq. (\ref{R(z)}). \end{proof}
\begin{remark} Clearly $R(0)=R$. Moreover, when $r=q$ and $s=q^{-1}$, the R-matrix
$R(z)$ turns into
\begin{equation*}
R_q(z)=(1-zq^2)\sum_{i=1}^nE_{ii}\otimes E_{ii}+(1-z)q\sum_{i \neq j}E_{ij}\otimes E_{ji}+(1-q^2)(\sum_{i>j}+z\sum_{i<j})E_{ii}\otimes E_{jj},
\end{equation*}
which is exactly the Jimbo R-matrix for the quantum affine algebra $U_{q}(\mathfrak{sl}_n)$.
In this regard we can view $R(z)$ as an $(r,s)$-analogue of the R-matrix $R_q(z)$ of the quantum affine algebra $U_{q}(\widehat{\mathfrak{sl}}_n)$ \cite{HRZ}.
\end{remark}

\section{$(r,s)$-wedge modules of $U_{r,s}(sl_n)$}

In \cite{BW3} Benkart and Witherspoon studied the $(r,s)$-symmetric tensor $S^2_{r,s}(V)$ (see Proposition \ref{prop2.3}) using the Hopf algebra structure. Here we give an alternative description of the symmetric tensors
using the fusion procedure as well as the anti-symmetric tensors.

\begin{proposition} The subspace
$S_{r,s}^2(V)$ is equal to the image of $R(rs^{-1})$ or $Ker\, R(r^{-1}s)$  on $V\otimes V$,
and $\Lambda_{r, s}^2(V)$ is equal to the kernel of $R(rs^{-1})$ or $Im\, R(r^{-1}s)$  on $V\otimes V$.
\end{proposition}
\begin{proof} This is a simple consequence of the spectral decomposition
of the $R$-matrix $R(z)$. In fact applying $R(rs^{-1})$ to $v_i\otimes v_j$ proves that
$S^2_{r, s}(V)=Im R(rs^{-1})$. Then one verifies that $R(rs^{-1})(v_i\otimes v_j-rv_j\otimes v_i)=0$ for $i<j$,
i.e., $KerR(rs^{-1})\supseteq\Lambda_{r, s}^2(V)$. The equality follows due to
$S^2_{r, s}\cap \Lambda^2_{r,s}=\{0\}$. The other relations are checked similarly.
\end{proof}

\begin{remark}
The above proposition is the analogue of the one-parameter case (see \cite{JMO}).
In fact, when  $s^{-1}=q=r$, the (r,s)-symmetric tensor $S^2_{r,s}(V)$ of $U_{r,s}(\mathfrak{sl}_n)$
reduces to the $q$-symmetric tensor $W$ of $U_q(\mathfrak{sl}_n)$ (see \cite{JMO}).
\end{remark}

Now we can study the general wedge products of the two-parameter case using the fusion procedure
as in \cite{JMO}. The following theorem completely determines $(r,s)$-wedge modules for $U_{r,s}(\mathfrak{sl}_n)$.

\begin{theorem} On the tensor product of the natural representation of $U_{r,s}(\mathfrak{sl}_n)$ we have
\begin{equation}
V^{\otimes k}/\sum_{i=0}^{k-2}V^{\otimes i}\otimes S^2_{r,s}(V) \otimes V^{\otimes (k-i-2)}\cong V(\overline{\omega_k})
\end{equation}
\end{theorem}

\begin{proof}
We denote the image of $v_{i_1}\otimes v_{i_2}...\otimes v_{i_k}\in V^{\otimes k}$ in $V^{\otimes k}/\sum_{i=0}^{k-2}V^{\otimes i}\otimes S^2_{r,s}(V) \otimes V^{\otimes (k-i-2)}$ by $v_{i_1}\wedge v_{i_2}\wedge \cdots\wedge v_{i_k}$.

First we show that the vector $v_{1}\wedge v_{2}\wedge \cdots\wedge v_{k}$ is a highest vector with highest weight $\overline{\omega_k}$ in the $U_{r,s}(sl_n)$-module.

It follows from Eqs. (\ref{eq2.2})--(\ref{eq2.3}) and the Hopf algebra structure of $U_{r,s}(\mathfrak{sl}_n)$
that
\begin{align*}
\omega_i\cdot v_1\wedge v_2\wedge\cdots\wedge v_k&=r^{<\varepsilon_i,\varepsilon_1+\varepsilon_2+\cdots+\varepsilon_k>}s^{<\varepsilon_{i+1},
\varepsilon_1+\varepsilon_2+\cdots+\varepsilon_k>}v_1\wedge v_2\wedge\cdots\wedge v_k, \\
\omega'_i\cdot v_1\wedge v_2\wedge\cdots\wedge v_k &= r^{<\varepsilon_{i+1},\varepsilon_1+\varepsilon_2+\cdots+\varepsilon_k>}s^{<\varepsilon_{i},\varepsilon_1+
\varepsilon_2+\cdots+\varepsilon_k>}v_1\wedge v_2\wedge\cdots\wedge v_k.
\end{align*}
Hence $v_1\wedge v_2\wedge\cdots\wedge v_k$ has weight $\overline{\omega_k}=\varepsilon_1+\varepsilon_2+\cdots+\varepsilon_k$.

Furthermore the action of $e_i$ on $V^{\otimes k}$ can be computed as follows. For example
\begin{align*}
&e_i\cdot v_{1}\otimes v_{2}\otimes\cdots\otimes v_{k}=\sum_{j=1}^k \omega_i v_1\otimes\cdots\otimes \omega_i v_{j-1}\otimes e_i v_j
\otimes v_{j+1}\otimes\cdots\otimes v_k\\
&=\sum_{j=1}^{k}r^{<\varepsilon_i,\varepsilon_1+\cdots+
\varepsilon_{j-1}>}s^{<\varepsilon_{i+1},\varepsilon_1+\cdots+\varepsilon_{j-1}>}
\delta_{j,i+1}v_1\cdots\otimes  v_{j-1}\otimes v_{i}
\otimes v_{j+1}\cdots\otimes v_k\\
&=rv_1\otimes\cdots\otimes  v_{i}\otimes v_{i}
\otimes v_{i+2}\otimes v_{i+3}\otimes\cdots\otimes v_k\quad \mbox{or} \quad 0\\
&\in \sum_{j=0}^{k-2}V^{\otimes j}\otimes S^2_{r,s}(V) \otimes V^{\otimes (k-j-2)}.
\end{align*}
Thus $e_i\cdot v_1\wedge v_2\wedge\cdots\wedge v_k=0$ in the $U_{r,s}(\mathfrak{sl}_n)$ module
$V^{\otimes k}/\sum_{i=0}^{k-2}V^{\otimes i}\otimes S^2_{r,s}(V) \otimes V^{\otimes (k-i-2)}$.

From the above discussion, we have that the fundamental representation $V(\overline{\omega}_k)$ is isomorphic to
a submodule of $V^{\otimes k}/\sum_{i=0}^{k-2}V^{\otimes i}\otimes S^2_{r,s}(V) \otimes V^{\otimes (k-i-2)}$.
By taking special values of $r$ and $s$, one sees that the two modules have the same dimension, so they are
isomorphic. \end{proof}

\medskip

\centerline{\bf Acknowledgments}
NJ gratefully acknowledges the partial support of
Simons Foundation grant 198129 and NSFC grant 11271138 during this work.

\bibliographystyle{amsalpha}

\end{document}